\documentclass{amsart}
\usepackage{amsmath,amssymb,amsthm}
\usepackage{mathtools}

\usepackage[driver=dvips,text={372pt,533pt},centering]{geometry}

\numberwithin{equation}{section}

\theoremstyle{plain}
\newtheorem{thm}{Theorem}[section]
\newtheorem{lem}[thm]{Lemma}
\newtheorem{prop}[thm]{Proposition}

\theoremstyle{definition}
\newtheorem{defn}[thm]{Definition}

\newtheorem{rem}[thm]{Remark}

\allowdisplaybreaks[1]

\newcommand{\Z}{\mathbb{Z}}
\newcommand{\Q}{\mathbb{Q}}

\newcommand{\C}{\mathbb{C}}

\newcommand{\bk}{{\boldsymbol{k}}}
\newcommand{\tk}{\tilde{k}}
\newcommand{\tbk}{\tilde{\boldsymbol{k}}}

\newcommand{\tbl}{\tilde{\boldsymbol{l}}}
\newcommand{\cA}{\mathcal{A}}

\newcommand{\Li}{\mathrm{Li}}
\newcommand{\tLi}{\widetilde{\mathrm{Li}}}
\newcommand{\abs}[1]{\lvert #1 \rvert}

\newcommand{\argstack}[2]{\genfrac{}{}{0pt}{}{#1}{#2}}

\title[Duality of multiple polylogarithms and their $q$-analogues]
{Duality of one-variable multiple polylogarithms and their $q$-analogues}
\author{Shuji Yamamoto}
\address{Department of Mathematics, Faculty of Science and Technology, Keio University, 
3-14-1 Hiyoshi, Kohoku-ku, Yokohama, 223-8522, Japan}
\email{yamashu@math.keio.ac.jp}
\thanks{This research was supported in part by JSPS KAKENHI Grant Numbers 
16H06336, 18K03221, 18H05233, 21K03185.}

\begin{document}
\maketitle

\begin{abstract}
The duality relation of one-variable multiple polylogarithms was proved by 
Hirose, Iwaki, Sato and Tasaka by means of iterated integrals. 
In this paper, we give a new proof using the method of connected sums, 
which was recently invented by Seki and the author. 
Interestingly, the connected sum involves the hypergeometric function in its connector. 
Moreover, we introduce two kinds of $q$-analogues of the one-variable multiple polylogarithms 
and generalize the duality to them. 
\end{abstract}

\section{Introduction}
For a tuple of positive integers $\bk=(k_1,\ldots,k_r)$ with $k_r\ge 2$, let 
\[\zeta(\bk)\coloneqq \sum_{0<m_1<\dots<m_r}\frac{1}{m_1^{k_1}m_2^{k_2}\dots m_r^{k_r}} \] 
be the multiple zeta value of index $\bk$. There are many studies on the $\Q$-linear relations 
among multiple zeta values. A remarkable example of such relations is the duality relation 
\begin{equation}\label{eq:duality MZV}
\zeta(\bk)=\zeta(\bk^\dagger). 
\end{equation}
Here, by representing the given index $\bk$ as 
\[\bk=(\underbrace{1,\ldots,1}_{a_1-1},b_1+1,\dots,\underbrace{1,\ldots,1}_{a_s-1},b_s+1)\]
with positive integers $a_1,b_1,\ldots,a_s,b_s$, 
its \emph{dual index} $\bk^\dagger$ is defined by 
\[\bk^\dagger\coloneqq 
(\underbrace{1,\ldots,1}_{b_s-1},a_s+1,\dots,\underbrace{1,\ldots,1}_{b_1-1},a_1+1).\]
As is well known (cf.~\cite[\S9]{Zag}), 
the duality relation \eqref{eq:duality MZV} is an immediate consequence of 
the iterated integral expression of multiple zeta values. 
Recently, Seki and the author \cite{SY1} gave a new method of proving the duality, 
which used certain multiple sums (called the \emph{connected sums}) 
and made no use of the integrals. 
An advantage of this method is that it can be directly generalized to obtain a $q$-analogue. 
In the present paper, we apply it to multiple polylogarithms and their $q$-analogues. 

For an index $\bk=(k_1,\ldots,k_r)$, not necessarily with $k_r\ge 2$, let 
\[\Li_\bk(z_1,\ldots,z_r)\coloneqq \sum_{0<m_1<\dots<m_r}
\frac{z_1^{m_1}z_2^{m_2-m_1}\dots z_r^{m_r-m_{r-1}}}{m_1^{k_1}m_2^{k_2}\dots m_r^{k_r}}\]
be the multiple polylogarithm.
Moreover, let $I$ be a subset of $\{1,\ldots,r\}$ and assume that $k_r\ne 1$ or $r\in I$. 
Then we consider a function $\Li^I_\bk(z)$ of a complex variable $z$ with $\abs{z}<1$ defined by  
\begin{equation}\label{eq:Li^I_bk}
\begin{split}
\Li^I_\bk(z)&\coloneqq\sum_{0=m_0<m_1<\dots<m_r}
\frac{z^{\sum_{i\in I}(m_i-m_{i-1})}}{m_1^{k_1}m_2^{k_2}\dots m_r^{k_r}} \\
&=\Li_\bk(z_1,\ldots,z_r) \quad \text{where}\quad 
z_i=\begin{cases} z & (i\in I),\\ 1 & (i\notin I). \end{cases}
\end{split}
\end{equation}
We call this function a \emph{one-variable multiple polylogarithm} or one-variable MPL for short. 
This name suggests that it is a function of \emph{one variable} $z$, and also that 
\emph{one} or \emph{the variable} $z$ is assigned to each of $r$ arguments of $\Li_\bk(z_1,\ldots,z_r)$. 
Note that we recover the multiple zeta values by putting $I=\emptyset$ (the empty set): 
\[\Li_\bk^\emptyset(z)=\zeta(\bk). \] 

Just as in the case of multiple zeta values, we are interested in 
$\Q$-linear relations among one-variable MPLs. 
For example, Hirose--Iwaki--Sato--Tasaka \cite{HIST} obtained such relations 
which generalize the duality and sum formulas for multiple zeta values, 
by using iterated integral representation of one-variable MPLs. 

Let us recall the duality formula of Hirose--Iwaki--Sato--Tasaka. 
Let $\cA\coloneqq\Q\langle e_0,e_1,e_z\rangle$ be the non-commutative polynomial algebra 
in three indeterminates $e_0$, $e_1$ and $e_z$, and define its linear subspace 
$\cA^0\coloneqq \Q+\Q e_z+e_1\cA e_0+e_1\cA e_z+e_z\cA e_0+e_z\cA e_z$. 
Then we define a linear map $L$ from $\cA^0$ to the space of holomorphic functions 
on the unit disc in $\C$ by $L(1)\coloneqq 1$ and 
\[L(e_{z_1}e_0^{k_1-1}\dots e_{z_r}e_0^{k_r-1})\coloneqq (-1)^r\Li^I_\bk(z)\]
with $z_i$ as in \eqref{eq:Li^I_bk}. 

\begin{thm}[{Duality of one-variable MPLs, \cite[Theorem 1.1]{HIST}}]\label{thm:Main}
Let $\tau\colon\cA\to\cA$ be the anti-automorphism of $\Q$-algebra defined by 
\[\tau(e_0)=e_z-e_1,\quad \tau(e_1)=e_z-e_0,\quad \tau(e_z)=e_z. \]
Then we have 
\begin{equation}\label{eq:Main}
L(w)=L(\tau(w))
\end{equation}
for any $w\in\cA^0$. 
\end{thm}

\begin{rem}
In \cite{HIST}, $L$ is defined as a map from $\cA^0$ to 
the space of holomorphic functions on $\C\setminus[0,1]$ 
(note that our variable $z$ corresponds to their $z^{-1}$). 
Analytic continuation of $\Li^I_\bk(z)$ to $\C\setminus[1,\infty)$ 
is obtained by using an iterated integral expression. 
\end{rem}

Note that, unlike \eqref{eq:duality MZV}, 
Theorem \ref{thm:Main} does not give a one-to-one identity of one-variable MPLs in general. 
For instance, the relation 
\[L(e_1e_0^2)=L((e_z-e_1)^2(e_z-e_0)) \qquad \text{(\eqref{eq:Main} for $w=e_1e_0^2$)}\]
amounts to the linear relation 
\begin{align*}
\zeta(3)&=\Li_{1,1,1}(z,z,z)-\Li_{1,1,1}(z,1,z)-\Li_{1,1,1}(1,z,z)+\Li_{1,1,1}(1,1,z)\\
&\quad+\Li_{1,2}(z,z)-\Li_{1,2}(z,1)-\Li_{1,2}(1,z)+\zeta(1,2) 
\end{align*}
among nine one-variable MPLs. 
In fact, as is done below, it is possible to rephrase Theorem \ref{thm:Main} as a one-to-one identity 
just like \eqref{eq:duality MZV} by considering appropriate linear combinations of one-variable MPLs. 
This reformulation is necessary for our proof of Theorem \ref{thm:Main}. 

We call a pair $\tk=(k,\mu)\in\Z_{>0}\times\{0,1\}$ an \emph{augmented positive integer}, 
and a tuple $\tbk=(\tk_1,\ldots,\tk_r)$ of augmented positive integers an \emph{augmented index}. 
Such $\tbk$ is said \emph{admissible} if $r=0$ (we write $\tbk=\varnothing$ in this case), 
or $r>0$ and $\tk_r\ne(1,1)$. 
For an augmented index $\tbk=\bigl((k_1,\mu_1),\ldots,(k_r,\mu_r)\bigr)$, 
we define $w(\tbk)\in\cA$ by 
\[w(\tbk)\coloneqq y_{\mu_1}x^{k_1-1}\dots y_{\mu_r}x^{k_r-1}, \]
where 
\[x\coloneqq e_0,\quad y_0\coloneqq -e_z,\quad y_1\coloneqq e_z-e_1. \]
For $\tbk=\varnothing$, we understand $w(\varnothing)=1$. 
Then it is easy to see that $w(\tbk)$, with $\tbk$ running over all admissible augmented indices, 
form a basis of $\cA^0$. Moreover, for any admissible augmented index $\tbk$, 
there exists a unique admissible augmented index $\tbk^\dagger$, 
called the \emph{dual} of $\tbk$, such that $\tau\bigl(w(\tbk)\bigr)=w(\tbk^\dagger)$. 

For a non-empty admissible index $\tbk=\bigl((k_1,\mu_1),\ldots,(k_r,\mu_r)\bigr)$, 
we define 
\[\tLi(\tbk;z)\coloneqq \sum_{0=m_0<m_1<\dots<m_r} 
\prod_{i=1}^r\frac{\mu_i+(-1)^{\mu_i}z^{m_i-m_{i-1}}}{m_i^{k_i}}, \]
and set $\tLi(\varnothing;z)\coloneqq 1$. 
Then the map $L$ satisfies (and is determined by) the formula 
\[L\bigl(w(\tbk)\bigr)=\tLi(\tbk;z), \]
and hence Theorem \ref{thm:Main} is equivalent to the following: 

\begin{thm}\label{thm:Main2}
For any admissible augmented index $\tbk$, we have 
\[\tLi(\tbk;z)=\tLi(\tbk^\dagger;z). \]
\end{thm}

\begin{rem}
Kaneko and Tsumura \cite{KT} study a version of multiple zeta values of level two defined by 
\[T(k_1,\ldots,k_r)\coloneqq 2^r\sum_{\substack{0<m_1<\dots<m_r\\ m_i\equiv i\bmod 2}}
\frac{1}{m_1^{k_1}\dots m_r^{k_r}}. \]
In our notation, this can be written as 
\[T(k_1,\ldots,k_r)=\tLi\bigl((k_1,1),\ldots,(k_r,1);-1\bigr), \]
and Theorem \ref{thm:Main2} is a generalization of the duality of $T$-values 
\cite[Theorem 3.1]{KT}. 
\end{rem}

We prove Theorem \ref{thm:Main2} in \S2. 
In \S3, we introduce two kinds of $q$-analogues of $\tLi(\tbk;z)$, 
and prove for each of them a generalization of Theorem \ref{thm:Main2}. 

\section{The proof of Theorem \ref{thm:Main2} via connected sums}
Let $F\bigl(\argstack{\alpha,\,\beta}{\gamma};z\bigr)$ denote the hypergeometric series 
\[F\biggl(\argstack{\alpha,\,\beta}{\gamma};z\biggr)
\coloneqq \sum_{n=0}^\infty\frac{(\alpha)_n\,(\beta)_n}{(\gamma)_n\,n!}z^n\]
where $(\alpha)_n\coloneqq \alpha(\alpha+1)\dots(\alpha+n-1)$ is the rising factorial. 
We use the following contiguous relations. 

\begin{lem}\label{lem:HG}
The hypergeometric series satisfies the identities 
\begin{align}
\label{eq:HG 1}
F\biggl(\argstack{\alpha,\,\beta}{\gamma};z\biggr)
&=F\biggl(\argstack{\alpha,\,\beta+1}{\gamma};z\biggr)
-z\,\frac{\alpha}{\gamma} F\biggl(\argstack{\alpha+1,\,\beta+1}{\gamma+1};z\biggr), \\
\label{eq:HG 2}
(\gamma-\alpha)F\biggl(\argstack{\alpha,\,\beta}{\gamma+1};z\biggr)
&=\gamma F\biggl(\argstack{\alpha,\,\beta}{\gamma};z\biggr)
-\alpha F\biggl(\argstack{\alpha+1,\,\beta}{\gamma+1};z\biggr). 
\end{align}
\end{lem}
\begin{proof}
These are well known identities, see e.g.\ the equations (9.2.13) and (9.2.5) in \cite{Leb}. 
Indeed, they are shown by comparing the coefficients of $z^n$, that is, by checking 
\begin{align*}
\frac{(\alpha)_n\,(\beta)_n}{(\gamma)_n\,n!}
&=\frac{(\alpha)_n\,(\beta+1)_n}{(\gamma)_n\,n!}
-\frac{\alpha}{\gamma}\frac{(\alpha+1)_{n-1}\,(\beta+1)_{n-1}}{(\gamma+1)_{n-1}\,(n-1)!},\\
(\gamma-\alpha)\frac{(\alpha)_n\,(\beta)_n}{(\gamma+1)_n\,n!}
&=\gamma\frac{(\alpha)_n\,(\beta)_n}{(\gamma)_n\,n!}
-\alpha\frac{(\alpha+1)_n\,(\beta)_n}{(\gamma+1)_n\,n!}, 
\end{align*}
respectively. 
\end{proof}

From now on, we assume that $z$ is a real number with $0\leq z<1$, 
which is harmless to the proof of Theorem \ref{thm:Main2}. 
The following definition is the core of our proof of Theorem \ref{thm:Main2}. 

\begin{defn}
Let $\tbk=\bigl((k_1,\mu_1),\ldots,(k_r,\mu_r)\bigr)$ and 
$\tbl=\bigl((l_1,\nu_1),\ldots,(l_s,\nu_s)\bigr)$ be augmented indices (admissible or not). 
Then we define 
\begin{multline*}
\tLi(\tbk;\tbl;z)\coloneqq \sum_{\substack{0=m_0<m_1<\dots<m_r\\ 0=n_0<n_1<\dots<n_s}}
\prod_{i=1}^r\frac{\mu_i+(-1)^{\mu_i}z^{m_i-m_{i-1}}}{m_i^{k_i}}\\
\times\prod_{j=1}^s\frac{\nu_j+(-1)^{\nu_j}z^{n_j-n_{j-1}}}{n_j^{l_j}} C(m_r,n_s;z),
\end{multline*}
where 
\[C(m,n;z)\coloneqq \frac{m!\;n!}{(m+n)!}F\biggl(\argstack{m,\ n}{m+n+1};z\biggr). \]
By abuse of notation, we also write $\tLi\bigl(w(\tbk);w(\tbl);z\bigr)$ for $\tLi(\tbk;\tbl;z)$. 
For example, $\tLi(y_1x;y_0y_1;z)$ means $\tLi\bigl((2,1);(1,0),(1,1);z\bigr)$. 
\end{defn}

The value of the series $\tLi(\tbk;\tbl;z)$ is always well defined 
as a non-negative real number or the positive infinity since, by the assumption $0\leq z<1$, 
all terms are non-negative real numbers. 
As in \cite{SY1}, we call this series $\tLi(\tbk;\tbl;z)$ the \emph{connected sum} 
and the factor $C(m_r,n_s;z)$ the \emph{connector}: 
Notice that, without the connector, 
the sum splits into two independent sums with respect to $m_i$'s and $n_j$'s. 
The connected sum satisfies the obvious symmetry 
\begin{equation}\label{eq:symmetry}
\tLi(\tbk;\tbl;z)=\tLi(\tbl;\tbk;z), 
\end{equation}
the \emph{boundary condition} 
\begin{equation}\label{eq:boundary}
\tLi(\tbk;\varnothing;z)=\tLi(\tbk;z)=\tLi(\varnothing;\tbk;z), 
\end{equation}
and the following \emph{transport relations}: 

\begin{prop}\label{prop:transport}
For any augmented indices $\tbk$ and $\tbl$, we have 
\begin{equation}\label{eq:transport 1}
\tLi\bigl(w(\tbk)y_0;w(\tbl);z\bigr)=\tLi\bigl(w(\tbk);w(\tbl)y_0;z\bigr). 
\end{equation}
If $\tbl$ is non-empty, we also have 
\begin{equation}\label{eq:transport 2}
\tLi\bigl(w(\tbk)y_1;w(\tbl);z\bigr)=\tLi\bigl(w(\tbk);w(\tbl)x;z\bigr).  
\end{equation}
\end{prop}
\begin{proof}
In order to prove the identity \eqref{eq:transport 1}, it suffices to show 
\begin{equation}\label{eq:connector 1}
\sum_{a>m} \frac{z^{a-m}}{a}C(a,n;z)=\sum_{b>n} \frac{z^{b-n}}{b}C(m,b;z)
\end{equation}
for any integers $m,n\geq 0$. 
By using \eqref{eq:HG 1}, we compute the left hand side of \eqref{eq:connector 1} as follows: 
\begin{align}
\notag&\sum_{a>m} z^{a-m}\frac{(a-1)!\,n!}{(a+n)!}F\biggl(\argstack{a,\,n}{a+n+1};z\biggr)\\
\notag&=\sum_{a>m}^\infty z^{a-m}\frac{(a-1)!\,n!}{(a+n)!}
\Biggl\{F\biggl(\argstack{a,\,n+1}{a+n+1};z\biggr)
-z\frac{a}{a+n+1}F\biggl(\argstack{a+1,\,n+1}{a+n+2};z\biggr)\Biggr\}\\
\notag&=\sum_{a>m}\Biggl\{z^{a-m}\frac{(a-1)!\,n!}{(a+n)!}F\biggl(\argstack{a,\,n+1}{a+n+1};z\biggr)\\
\notag&\hspace{140pt}
-z^{a+1-m}\frac{a!\,n!}{(a+1+n)!}F\biggl(\argstack{a+1,\,n+1}{a+1+n+1};z\biggr)\Biggr\}\\
\label{eq:connector 1.5}
&=z\frac{m!\,n!}{(m+n+1)!}F\biggl(\argstack{m+1,\,n+1}{m+n+2};z\biggr). 
\end{align}
Since the last expression is symmetric with respect to $m$ and $n$, 
we obtain \eqref{eq:connector 1}. 

In order to prove the identity \eqref{eq:transport 2}, it suffices to show 
\begin{equation}\label{eq:connector 2}
\sum_{a>m} \frac{1-z^{a-m}}{a}C(a,n;z)=\frac{1}{n}C(m,n;z)
\end{equation}
for any integers $m\geq 0$ and $n>0$. By using \eqref{eq:HG 2}, we see that 
\begin{align*}
&\sum_{a>m} \frac{1}{a}C(a,n;z)\\
&=\sum_{a>m}\frac{(a-1)!\,n!}{(a+n)!}F\biggl(\argstack{a,\,n}{a+n+1};z\biggr)\\
&=\sum_{a>m}\frac{(a-1)!\,(n-1)!}{(a+n)!}\Biggr\{(a+n)F\biggl(\argstack{a,\,n}{a+n};z\biggr)
-aF\biggl(\argstack{a+1,\,n}{a+n+1};z\biggr)\Biggr\}\\
&=\sum_{a>m}\Biggr\{\frac{(a-1)!\,(n-1)!}{(a-1+n)!}F\biggl(\argstack{a,\,n}{a+n};z\biggr)
-\frac{a!\,(n-1)!}{(a+n)!}F\biggl(\argstack{a+1,\,n}{a+n+1};z\biggr)\Biggr\}\\
&=\frac{m!\,(n-1)!}{(m+n)!}F\biggl(\argstack{m+1,\,n}{m+n+1};z\biggr). 
\end{align*}
This and \eqref{eq:connector 1.5} imply that 
\begin{align*}
&\sum_{a>m} \frac{1-z^{a-m}}{a}C(a,n;z)\\
&=\frac{m!\,(n-1)!}{(m+n)!}F\biggl(\argstack{m+1,\,n}{m+n+1};z\biggr)
-z\frac{m!\,n!}{(m+n+1)!}F\biggl(\argstack{m+1,\,n+1}{m+n+2};z\biggr)\\
&=\frac{m!\,(n-1)!}{(m+n)!}F\biggl(\argstack{m,\,n}{m+n+1};z\biggr)=\frac{1}{n}C(m,n;z), 
\end{align*}
by \eqref{eq:HG 1} again. Thus we obtain \eqref{eq:connector 2}. 
\end{proof}

Now we turn to the proof of  Theorem \ref{thm:Main2}. 

\begin{proof}[Proof of Theorem \ref{thm:Main2}]
By the transport relations \eqref{eq:transport 1} and \eqref{eq:transport 2} 
combined with the symmetry \eqref{eq:symmetry}, we have 
\[\tLi\bigl(w(\tbk)u;w(\tbl);z\bigr)=\tLi\bigl(w(\tbk);w(\tbl)\tau(u);z\bigr)\]
for $u\in\{x,y_0,y_1\}$. By using it repeatedly, we obtain 
\[\tLi\bigl(w(\tbk);1;z\bigr)=\tLi\bigl(1;\tau(w(\tbk));z\bigr) \]
for any admissible augmented index $\tbk$. 
By the boundary condition \eqref{eq:boundary}, this is exactly what we have to show. 
\end{proof}

\begin{rem}
As a by-product of the above discussion, one sees that 
\[\tLi\bigl(\tbk;\tbl;z)=\infty\iff 
\text{one of $\tbk$ and $\tbl$ is empty and the other is not admissible. }\]
Indeed, the case in which $\tbk$ or $\tbl$ is empty is easy. 
When both $\tbk$ and $\tbl$ are non-empty, one can use transport relations 
to obtain $\tLi(\tbk;\tbl;z)=\tLi(\varnothing;\tbl{}';z)$ for some admissible $\tbl{}'$, 
which shows that $\tLi(\tbk;\tbl;z)<\infty$. 
\end{rem}

\section{$q$-analogues}
There are quite many versions of $q$-analogues of multiple zeta values 
and results on them (see \cite[Chap.~12]{ZhaoBook} for a general account). 
The duality relation and its generalization to the Ohno-type sums for the Bradley--Zhao model 
\[\zeta_q^{\mathrm{BZ}}(\bk)\coloneqq 
\sum_{0<m_1<\dots<m_r}\frac{q^{(k_1-1)m_1+\dots+(k_r-1)m_r}}{[m_1]^{k_1}\dots[m_r]^{k_r}}\]
($[m]=[m]_q\coloneqq (1-q^m)/(1-q)$ denotes the $q$-integer) 
was proved by Bradley \cite[Theorem 5]{Bra}, and 
its proof via connected sums was given in \cite{SY1}. 
More recently, Brindle \cite{Bri} further developed the latter proof to 
include the duality \cite[Theorem 12.3.2]{ZhaoBook} for the Schlesinger--Zudilin model 
\[\zeta_q^{\mathrm{SZ}}(\bk)\coloneqq 
\sum_{0<m_1<\dots<m_r}\frac{q^{k_1m_1+\dots+k_rm_r}}{[m_1]^{k_1}\dots[m_r]^{k_r}}. \]

In this paper, we consider the following two kinds of $q$-analogues of $\tLi(\tbk;z)$. 

\begin{defn}
For a parameter $q$ with $0<q<1$ and an admissible augmented index 
$\tbk=\bigl((k_1,\mu_1),\ldots,(k_r,\mu_r)\bigr)$, we define  
\begin{align*}
\tLi_q^{(1)}(\tbk;z)&\coloneqq \sum_{0=m_0<m_1<\dots<m_r} 
\prod_{i=1}^r\frac{q^{(k_i-1)m_i}(\mu_i+(-1)^{\mu_i}q^{m_{i-1}}z^{m_i-m_{i-1}})}{[m_i]^{k_i}}, \\
\tLi_q^{(2)}(\tbk;z)&\coloneqq \sum_{0=m_0<m_1<\dots<m_r} 
\prod_{i=1}^r\frac{\mu_iq^{m_i}+(-1)^{\mu_i}z^{m_i-m_{i-1}}}{[m_i]^{k_i}}. 
\end{align*}
\end{defn}

Both of these $q$-analogues satisfy the same duality relation as Theorem \ref{thm:Main2}. 

\begin{thm}\label{thm:Main3}
For any admissible augmented index $\tbk$ and $\epsilon=1,2$, we have 
\[\tLi^{(\epsilon)}_q(\tbk;z)=\tLi^{(\epsilon)}_q(\tbk^\dagger;z). \]
\end{thm}

\begin{rem}
Let $\tbk=\bigl((k_1,\mu_1),\ldots,(k_r,\mu_r)\bigr)$ be an admissible augmented index 
with $\mu_i=1$ for all $i=1,\ldots,r$, and $\bk=(k_1,\ldots,k_r)$ the corresponding usual index. 
Then the dual $\tbk^\dagger$ of $\tbk$ has the same property and 
the corresponding index is $\bk^\dagger$, the dual of $\bk$. 
Note also that 
\begin{align*}
\tLi_q^{(1)}(\tbk;0)&=\sum_{0=m_0<m_1<\dots<m_r} \prod_{i=1}^r\frac{q^{(k_i-1)m_i}}{[m_i]^{k_i}}
=\zeta_q^{\mathrm{BZ}}(\bk),\\
\tLi_q^{(2)}(\tbk;0)&=\sum_{0=m_0<m_1<\dots<m_r} \prod_{i=1}^r\frac{q^{m_i}}{[m_i]^{k_i}}
=\zeta_q^{(1,\ldots,1)}(\bk), 
\end{align*}
where the last expression uses Zhao's general notation \cite[(11.1)]{ZhaoBook}. 
Therefore, Theorem \ref{thm:Main3} includes the duality relation for $\zeta_q^{\mathrm{BZ}}(\bk)$ 
and $\zeta_q^{(1,\ldots,1)}(\bk)$. 
\end{rem}

We prove Theorem \ref{thm:Main3} along the same line as Theorem \ref{thm:Main2}. 
Put $[n]!\coloneqq [1]\dots[n]$ and let $\phi_q\bigl(\argstack{\alpha,\beta}{\gamma};z\bigr)$ 
denote the $q$-hypergeometric series 
\[\phi_q\biggl(\argstack{\alpha,\beta}{\gamma};z\biggr)
\coloneqq \sum_{n=0}^\infty \frac{(\alpha;q)_n(\beta;q)_n}{(\gamma;q)_n(q;q)_n}z^n, \]
where $(\alpha;q)_n\coloneqq (1-\alpha)(1-q\alpha)\dots(1-q^{n-1}\alpha)$ is 
the $q$-shifted factorial. 
We need the following contiguous relations.

\begin{lem}
The $q$-hypergeometric series satisfies the identities 
\begin{align}
\label{eq:q1HG 1}
\phi_q\biggl(\argstack{\alpha,\,\beta}{\gamma};z\biggr)
&=\phi_q\biggl(\argstack{\alpha,\,q\beta}{\gamma};z\biggr)
-z\,\frac{(1-\alpha)\beta}{1-\gamma}\phi_q\biggl(\argstack{q\alpha,\,q\beta}{q\gamma};z\biggr), \\
\label{eq:q1HG 2}
(\alpha-\gamma)\phi_q\biggl(\argstack{\alpha,\,\beta}{q\gamma};z\biggr)
&=(1-\gamma)\alpha\,\phi_q\biggl(\argstack{\alpha,\,\beta}{\gamma};z\biggr)
-(1-\alpha)\gamma\,\phi_q\biggl(\argstack{q\alpha,\,\beta}{q\gamma};z\biggr), \\
\label{eq:q2HG 1}
\phi_q\biggl(\argstack{\alpha,\,\beta}{\gamma};qz\biggr)
&=\phi_q\biggl(\argstack{\alpha,\,q\beta}{\gamma};z\biggr)
-z\,\frac{1-\alpha}{1-\gamma} \phi_q\biggl(\argstack{q\alpha,\,q\beta}{q\gamma};z\biggr), \\
\label{eq:q2HG 2}
(\alpha-\gamma)\phi_q\biggl(\argstack{\alpha,\,\beta}{q\gamma};qz\biggr)
&=(1-\gamma)\phi_q\biggl(\argstack{\alpha,\,\beta}{\gamma};z\biggr)
-(1-\alpha)\phi_q\biggl(\argstack{q\alpha,\,\beta}{q\gamma};z\biggr). 
\end{align}
\end{lem}
\begin{proof}
Again, they are shown by comparing the coefficients of $z^n$. 
\end{proof}

\begin{defn}
We define the connected sums 
\begin{multline*}
\tLi_q^{(1)}(\tbk;\tbl;z)
\coloneqq \sum_{\substack{0=m_0<m_1<\dots<m_r\\ 0=n_0<n_1<\dots<n_s}}
\prod_{i=1}^r\frac{q^{(k_i-1)m_i}(\mu_i+(-1)^{\mu_i}q^{m_{i-1}}z^{m_i-m_{i-1}})}{[m_i]^{k_i}}\\
\times\prod_{j=1}^r\frac{q^{(l_j-1)n_j}(\nu_j+(-1)^{\nu_j}q^{n_{j-1}}z^{n_j-n_{j-1}})}{[n_j]^{l_j}}
C_q^{(1)}(m_r,n_s;z) 
\end{multline*}
and 
\begin{multline*}
\tLi_q^{(2)}(\tbk;\tbl;z)
\coloneqq \sum_{\substack{0=m_0<m_1<\dots<m_r\\ 0=n_0<n_1<\dots<n_s}}
\prod_{i=1}^r\frac{\mu_iq^{m_i}+(-1)^{\mu_i}z^{m_i-m_{i-1}}}{[m_i]^{k_i}}\\
\times\prod_{j=1}^r\frac{\nu_jq^{n_j}+(-1)^{\nu_j}z^{n_j-n_{j-1}}}{[n_j]^{l_j}}
C_q^{(2)}(m_r,n_s;z), 
\end{multline*}
with the connectors 
\[C_q^{(1)}(m,n;z)\coloneqq q^{mn}\frac{[m]!\,[n]!}{[m+n]!}\,
\phi_q\biggl(\argstack{q^m,q^n}{q^{m+n+1}};z\biggr) \]
and 
\[C_q^{(2)}(m,n;z)\coloneqq \frac{[m]!\,[n]!}{[m+n]!}\,
\phi_q\biggl(\argstack{q^m,q^n}{q^{m+n+1}};qz\biggr), \]
respectively. We also use the word notation 
\[\tLi_q^{(\epsilon)}\bigl(w(\tbk);w(\tbl);z\bigr)\coloneqq\tLi_q^{(\epsilon)}(\tbk;\tbl;z). \]
\end{defn}

The symmetry 
\begin{equation}\label{eq:q symmetry}
\tLi_q^{(\epsilon)}(\tbk;\tbl;z)=\tLi_q^{(\epsilon)}(\tbl;\tbk;z)
\end{equation}
and the boundary condition 
\begin{equation}\label{eq:q boundary}
\tLi_q^{(\epsilon)}(\tbk;\varnothing;z)=\tLi_q^{(\epsilon)}(\tbk;z)
=\tLi_q^{(\epsilon)}(\varnothing;\tbk;z)
\end{equation}
are obvious. Furthermore, we have the following transport relations. 

\begin{prop}\label{prop:q transport}
For any augmented indices $\tbk$ and $\tbl$ and $\epsilon=1,2$, we have 
\begin{equation}\label{eq:q transport 1}
\tLi_q^{(\epsilon)}\bigl(w(\tbk)y_0;w(\tbl);z\bigr)=\tLi_q^{(\epsilon)}\bigl(w(\tbk);w(\tbl)y_0;z\bigr). 
\end{equation}
If $\tbl$ is non-empty, we also have 
\begin{equation}\label{eq:q transport 2}
\tLi_q^{(\epsilon)}\bigl(w(\tbk)y_1;w(\tbl);z\bigr)=\tLi_q^{(\epsilon)}\bigl(w(\tbk);w(\tbl)x;z\bigr).  
\end{equation}
\end{prop}
\begin{proof}
First let $\epsilon=1$. The first relation \eqref{eq:q transport 1} follows from the identity 
\begin{equation}\label{eq:q connector 1.5}
\sum_{a>m}\frac{q^mz^{a-m}}{[a]}C_q^{(1)}(a,n;z)
=zq^{mn+m+n}\frac{[m]!\,[n]!}{[m+n+1]!}\phi_q\biggl(\argstack{q^{m+1},\,q^{n+1}}{q^{m+n+2}};z\biggr), 
\end{equation}
which is a symmetric expression with respect to $m$ and $n$. 
We can show \eqref{eq:q connector 1.5} by using \eqref{eq:q1HG 1} 
and making a telescoping sum. For the proof of the second relation \eqref{eq:q transport 2}, 
we use \eqref{eq:q1HG 2} and make a telescoping sum to see 
\[\sum_{a>m}\frac{1}{[a]}C_q^{(1)}(a,n;z)
=q^{(m+1)n}\frac{[m]!\,[n-1]!}{[m+n]!}\phi_q\biggl(\argstack{q^{m+1},\,q^n}{q^{m+n+1}};z\biggr). \]
Then, combining this with \eqref{eq:q connector 1.5}, and using \eqref{eq:q1HG 1} again, 
we obtain 
\[\sum_{a>m}\frac{1-q^mz^{a-m}}{[a]}C_q^{(1)}(a,n;z)=\frac{q^n}{[n]}C_q^{(1)}(m,n;z), \]
which implies \eqref{eq:q transport 2}. This completes the proof for $\epsilon=1$. 

Almost the same argument works for $\epsilon=2$ 
if we replace the formulas \eqref{eq:q1HG 1} and \eqref{eq:q1HG 2} 
with \eqref{eq:q2HG 1} and \eqref{eq:q2HG 2}, respectively. 
\end{proof}

Given the properties \eqref{eq:q symmetry}--\eqref{eq:q transport 2} of the connected sum, 
the proof of Theorem \ref{thm:Main3} is the same as that of Theorem \ref{thm:Main2}. 

\section*{Acknowledgments}
The author sincerely thanks the anonymous referee(s) for useful comments 
to improve the presentation.

	
\end{document}